\documentclass[reqno]{amsart}
\usepackage{amssymb,latexsym,amsmath,amsthm,enumerate,amsbsy}
\usepackage[mathscr]{eucal}
\usepackage{framed,color,graphicx}
\usepackage{mathrsfs}
\usepackage[all]{xy}
\usepackage{tikz}
\usepackage{cite}

\usetikzlibrary{positioning,decorations.pathreplacing,patterns,decorations.pathmorphing}
\tikzset{%
element/.style={draw, shape=circle, fill=white, inner sep=1.4pt}
}

\DeclareSymbolFont{bbold}{U}{bbold}{m}{n}
\DeclareSymbolFontAlphabet{\mathbbold}{bbold}

\theoremstyle{plain}
\newtheorem{thm}{Theorem}[section]
\newtheorem{lem}[thm]{Lemma}
\newtheorem{cor}[thm]{Corollary}
\newtheorem{pro}[thm]{Proposition}

\theoremstyle{definition}

\newcommand{\up}[1]{\textup{#1}}

\newcommand{\bp}{\mathbf{p}}
\newcommand{\bq}{\mathbf{q}}

\newcommand{\bu}{\mathbf{u}}
\newcommand{\bv}{\mathbf{v}}

\begin{document}

\title[Every ai-semiring satisfying $xy\approx xz$ is finitely based]
{Every additively idempotent semiring satisfying $xy\approx xz$ is finitely based}

\author{Mengya Yue}
\address{School of Mathematics, Northwest University, Xi'an, 710127, Shaanxi, P.R. China}
\email{myayue@yeah.net}

\author{Miaomiao Ren}
\address{School of Mathematics, Northwest University, Xi'an, 710127, Shaanxi, P.R. China}
\email{miaomiaoren@yeah.net}

\subjclass[2010]{16Y60, 03C05, 08B05}
\keywords{semiring, finitely based, finitely generated, Cross variety}
\thanks{Miaomiao Ren, corresponding author, is supported by National Natural Science Foundation of China (12371024, 12571020).
}

\begin{abstract}
We study the finite basis problem for additively idempotent semirings satisfying the identity $xy \approx xz$.
Let $\mathbf{R}$ denote the variety of all such semirings.
Yue et al. (2025, Algebra Universalis, DOI:10.1007/s00012-025-00908-5) established that $\mathbf{R}$ is finitely generated.
In this paper, we show that the subvariety lattice of $\mathbf{R}$ forms a distributive lattice of order $10$.
As a consequence, the variety $\mathbf{R}$ is a Cross variety, and every member of $\mathbf{R}$ is finitely based.
\end{abstract}

\maketitle

\section{Introduction and preliminaries}
A \emph{variety} is a class of algebras that is closed under taking subalgebras, homomorphic images, and arbitrary direct products.
By Birkhoff's celebrated theorem, a class of algebras is a variety if and only if it is an equational class, that is,
the class of all algebras that satisfy some set of identities.

Let $\mathcal{V}$ be a variety. Then $\mathcal{V}$ is \emph{finitely based}
if it has a finite equational basis, that is, it can be defined by finitely many identities.
Otherwise, $\mathcal{V}$ is \emph{nonfinitely based}.
The variety $\mathcal{V}$ is \emph{finitely generated}
if it can be generated by a finite number of finite algebras.
The variety $\mathcal{V}$ is a \emph{Cross variety}
if it is finitely based, finitely generated, and has a finite number of subvarieties.
According to \cite{mv}, it is known that every subvariety of a Cross variety is both finitely based and finitely generated.

An algebra $A$ is finitely based if the variety $\mathsf{V}(A)$ it generates is finitely based;
Otherwise, $A$ is nonfinitely based.
It is well-known that every finite group (finite ring) generates a Cross variety (see~\cite{rl, i, s}).
However, this property does not hold for all finite semigroups or finite semirings (see~\cite{d, vol01}).

An \emph{additively idempotent semiring} (or ai-semiring for short)
is an algebra $(S, +, \cdot)$ equipped with two binary operations $+$ and $\cdot$ satisfying the following axioms:
\begin{itemize}
\item The additive reduct $(S, +)$ forms a commutative idempotent semigroup;

\item The multiplicative reduct $(S, \cdot)$ forms a semigroup;

\item The distributive laws hold:
\[
x(y+z)\approx xy+xz,~\text{and}~(x+y)z\approx xz+yz.
\]
\end{itemize}
Such algebras have broad applications in various fields, including algebraic geometry, tropical algebraic geometry,
theoretical computer science, and information science (see~\cite{cc, gl, go, ms}).

Over the past decades,
the finite basis problem for ai-semirings has attracted significant attention from scholars worldwide
and has been intensively studied
(see~\cite{d, gmrz, gpz05, jrz, mr, pas05, rlzc, rlyc, rjzl, rzw, rzs20, shap23, sr, wrz, yrzs, zrc}).
Dolinka~\cite{d} found the first example of finite nonfinitely based ai-semirings, which contains $7$ elements.
Volkov~\cite{vol21} proved that the ai-semiring $B_2^1$ whose multiplicative reduct is a
$6$-element Brandt semigroup has no finite basis for its equational theory.
Shao and Ren~\cite{sr} established that every variety generated by ai-semirings of order two is finitely based.
Zhao et al.~\cite{zrc} showed that
with the possible exception of the semiring $S_7$ (its Cayley tables can be found in Table \ref{tb24111401}),
all ai-semirings of order three are finitely based.
Jackson et al.~\cite{jrz} presented some general results about the finite basis problem for finite ai-semirings,
and confirmed that $S_7$ is nonfinitely based.
Recently,
Gao et al.~\cite{gmrz}, Ren et al.~\cite{rlyc, rlzc}, Shaprynski\v{\i}~\cite{shap23},
Wu et al.~\cite{wrz}, and Yue et al.~\cite{yrzs}
initiated the study of the finite basis problem for ai-semirings of order four.
Up to now, this work is still ongoing.

\begin{table}[ht]
\caption{The Cayley tables of $S_7$} \label{tb24111401}
\begin{tabular}{c|ccc}
$+$      &$\infty$&$a$&$1$\\
\hline
$\infty$ &$\infty$&$\infty$&$\infty$\\
$a$      &$\infty$&$a$&$\infty$\\
$1$      &$\infty$&$\infty$&$1$\\
\end{tabular}\qquad
\begin{tabular}{c|ccc}
$\cdot$  &$\infty$&$a$&$1$\\
\hline
$\infty$ &$\infty$&$\infty$&$\infty$\\
$a$      &$\infty$&$\infty$&$a$\\
$1$      &$\infty$&$a$&$1$\\
\end{tabular}
\end{table}

On the other hand,
McKenzie and Romanowska~\cite{mr} proved that every ai-semiring satisfying the identities
$x^2 \approx x$ and $xy \approx yx$ is finitely based.
Pastijn et al.~\cite{gpz05, pas05} showed that there are exactly $78$ varieties
of ai-semirings satisfying the identity $x^2 \approx x$, each of which is finitely based.
Ren et al.~\cite{rz16, rzw} demonstrated that the set of all ai-semiring varieties satisfying the identity
$x^3 \approx x$ forms a $179$-element distributive lattice, with every member being finitely based.
Ren et al.~\cite{rzs20} established that if $n>1$ is an integer and $n-1$ is square-free,
then every ai-semiring satisfying the identities $x^n \approx x$ and $xy \approx yx$ is finitely based.

Despite the above progress on the finite basis problem for ai-semirings,
there is no analogue of Perkins' foundational results for semigroups \cite{per69}:
the finite basis property for commutative semigroups and uniformly periodic permutative semigroups has been established.
In fact, the paper \cite{jrz} demonstrates that even finite ai-semirings with commutative multiplicative reducts
may have no finite equational basis, with $S_7$ being a specific example.
This striking contrast motivates Jackson et al. \cite{rjzl} to advocate for
identifying general laws ensuring finite axiomatizability for varieties of ai-semirings.
In this paper, we advance this research program by showing that
every ai-semiring satisfying the identity
\begin{equation}\label{id0703}
xy\approx xz
\end{equation}
is finitely based.
Let $\mathbf{R}$ denote the ai-semiring variety defined by the identity \eqref{id0703}.
Then for any algebra $S$ in $\mathbf{R}$, the multiplicative Cayley table of $S$ has constant rows,
that is, all entries in each row are identical.
In section 2, we shall show that the subvariety lattice of $\mathbf{R}$ is a distributive lattice of
order $10$. Each member of this lattice is finitely based and finitely generated.

Let $S$ be an ai-semiring. Then the binary relation $\leq$ defined by
\[
a \leq b \Leftrightarrow a+b=b,
\]
is a partial order on $S$.
It is easy to see that $\leq$ is compatible with the addition and the multiplication of $S$.
Let $X$ denote a countably infinite set of variables and $X^+$ the free semigroup on $X$.
By distributivity, all ai-semiring terms over $X$ are finite sums of words in $X^+$.
An \emph{ai-semiring identity} over $X$ is an expression of the form
\[
\bu\approx \bv,
\]
where $\bu$ and $\bv$ are ai-semiring terms over $X$.
From \cite[Theorem 2.5]{ku} we know that the ai-semiring $(P_f(X^+), \cup, \cdot)$ consisting of all non-empty finite subsets of $X^+$
is free in the variety of all ai-semirings on $X$.
So we sometimes write
\[
\{\bu_i \mid 1 \leq i \leq k\}\approx \{\bv_j \mid 1 \leq j \leq \ell\}
\]
for the ai-semiring identity
\[
\bu_1+\cdots+\bu_k\approx \bv_1+\cdots+\bv_\ell.
\]
An \emph{ai-semiring substitution} is an endomorphism of $P_f(X^+)$.
Let $S$ be an ai-semiring and $\bu\approx \bv$ an ai-semiring identity.
We say that $S$ \emph{satisfies} $\bu\approx \bv$ or $\bu\approx \bv$ \emph{holds} in $S$
if $\varphi(\bu)=\varphi(\bv)$ for all semiring homomorphisms $\varphi: P_f(X^+)\rightarrow S$.

Suppose that $\Sigma$ is a set of ai-semiring identities that includes the identities defining the variety of all ai-semirings.
Let $\mathbf{u}\approx \mathbf{v}$ be an ai-semiring identity such that
\[
\mathbf{u}=\mathbf{u}_1+\cdots+\mathbf{u}_k,~\mathbf{v}=\mathbf{v}_1+\cdots+\mathbf{v}_\ell,
\]
where $\mathbf{u}_i,\mathbf{v}_j\in X^+$ for $1\leq i\leq k$ and $1\leq j\leq \ell$.
It is easy to verify that the ai-semiring variety defined by $\mathbf{u}\approx \mathbf{v}$
coincides with the ai-semiring variety defined by the simpler identities:
\[
\mathbf{u} \approx \mathbf{u}+\mathbf{v}_j,~\mathbf{v} \approx \mathbf{v}+\mathbf{u}_i
\]
for all $1\leq i\leq k$ and $1\leq j\leq \ell$.
Therefore, to prove that $\mathbf{u}\approx \mathbf{v}$ is derivable from $\Sigma$, it suffices to show that for each $1\leq i\leq k$ and $1\leq j\leq \ell$, the identities $\mathbf{u}\approx \mathbf{u}+\mathbf{v}_j$ and $\mathbf{v}\approx \mathbf{v}+\mathbf{u}_i$ can be derived from $\Sigma$. This technique will be frequently employed in what follows.

Next, we introduce some notation that will be repeatedly used in the sequel.
Let $\bp$ be a word in $X^+$ and $x$ a letter in $X$. Then
\begin{itemize}
\item $h(\bp)$ denotes the first variable that occurs in $\bp$.

\item $t(\bp)$ denotes the last variable that occurs in $\bp$.

\item $\ell(\bp)$ denotes the number of variables occurring in $\bp$ counting multiplicities.

\item $s(\bp)$ denotes the word obtained from $\bp$ by deleting the letter $h(\bp)$, that is, $\bp=h(\bp)s(\bp)$.
\end{itemize}

Up to isomorphism, there are exactly 6 ai-semirings of order two (see \cite{sr}).
In this work, we only need four of them, which are denoted by $L_2$, $R_2$, $N_2$, and $T_2$.
Their Cayley tables and equational bases are given in Table~\ref{2-element ai-semirings}.
We assume that the carrier set of each algebra is $\{0,1\}$.
It is easy to check that $L_2$, $N_2$, and $T_2$ lie in the variety ${\bf R}$.
The following result, adapted from \cite{sr}, characterizes the equational theories
of these semirings and will be used implicitly throughout.

\begin{lem}\label{nlemma1}
Let $\bu\approx \bu+\bq$ be a nontrivial ai-semiring identity such that
$\bu=\bu_1+\cdots+\bu_n$, where $\bu_i, \bq\in X^+$, $1\leq i \leq n$. Then
\begin{itemize}
\item[$(1)$] $\bu\approx \bu+\bq$ holds in $L_2$ if and only if $h(\bu_i)=h(\bq)$ for some $\bu_i \in \bu$.

\item[$(2)$] $\bu\approx \bu+\bq$ holds in $R_2$ if and only if $t(\bu_i)=t(\bq)$ for some $\bu_i \in \bu$.

\item[$(3)$] $\bu\approx \bu+\bq$ holds in $N_2$ if and only if $\ell(\bq)\geq 2$.

\item[$(4)$] $\bu\approx \bu+\bq$ holds in $T_2$ if and only if $\ell(\bu_i)\geq 2$ for some $\bu_i \in \bu$.
\end{itemize}
\end{lem}

\begin{table}[ht]
\caption{Some 2-element ai-semirings}
\label{2-element ai-semirings}
\begin{tabular}{cccccc}
\hline
Semiring&   $+$ &  $\cdot$ & Equational basis\\
\hline
$L_2$&   \begin{tabular}{cc}
                    0 & 1   \\
                    1 & 1
              \end{tabular}&
\begin{tabular}{cc}
                    0 & 0  \\
                    1 & 1
              \end{tabular}

&
\begin{tabular}{cc}
                    & $xy\approx x$
\end{tabular}

\\
\hline
$R_2$&   \begin{tabular}{cc}
                    0 & 1   \\
                    1 & 1
              \end{tabular}&
\begin{tabular}{cc}
                    0 & 1   \\
                    0 & 1
              \end{tabular}
&
\begin{tabular}{cc}
                    & $xy\approx y$
\end{tabular}
\\
\hline
$N_2$&   \begin{tabular}{cc}
                    0 & 1   \\
                    1 & 1
              \end{tabular}&
\begin{tabular}{cc}
                    0 & 0  \\
                    0 & 0
              \end{tabular}
&
\begin{tabular}{cc}
                    & $x_1x_2\approx y_1y_2$, $x\approx x^2+x$
\end{tabular}
\\
\hline
$T_2$&   \begin{tabular}{cc}
                    0 & 1   \\
                    1 & 1
              \end{tabular}&
\begin{tabular}{cc}
                    1 & 1 \\
                    1 & 1
              \end{tabular}
&
\begin{tabular}{cc}
                    & $x_1x_2\approx y_1y_2$, $x^2\approx x^2+x$
\end{tabular}
\\
\hline

\end{tabular}

\end{table}

Up to isomorphism, there are exactly $61$ ai-semirings of order three,
which are denoted by $S_i, 1\leq i\leq 61$.
A complete description of these algebras can be found in \cite{zrc}.
For our current work, we specifically require the semirings
$S_{56}$, and $S_{58}$.
Table \ref{3-element ai-semirings} presents their Cayley tables and equational bases.
For uniformity,
we assume that the carrier set of each of these semirings is $\{1, 2, 3\}$.
It is easy to see that $S_{58}$ is a member of the variety ${\bf R}$.

\begin{table}[htbp]
\caption{Some 3-element ai-semirings}\label{3-element ai-semirings}
\begin{tabular}{cccc}
\hline
Semiring&   $+$ &  $\cdot$ & Equational basis\\
\hline
$S_{56}$&   \begin{tabular}{ccc}
                    1 & 1 & 3  \\
                    1 & 2 & 3\\
                    3 & 3 & 3
              \end{tabular}&
\begin{tabular}{ccc}
                    3 & 2 & 3  \\
                    3 & 2 & 3\\
                    3 & 2 & 3
              \end{tabular}
&
\begin{tabular}{cc}
                    & $xy\approx zy$, $x^2\approx x^2+x$
\end{tabular}
\\
\hline
$S_{58}$&   \begin{tabular}{ccc}
                    1 & 1 & 3  \\
                    1 & 2 & 3\\
                    3 & 3 & 3
              \end{tabular}&
\begin{tabular}{ccc}
                    3 & 3 & 3  \\
                    2 & 2 & 2\\
                    3 & 3 & 3
              \end{tabular}
&
\begin{tabular}{cc}
                    & $xy\approx xz$, $x^2\approx x^2+x$
\end{tabular}
\\
\hline
\end{tabular}
\end{table}

Up to isomorphism, there are exactly $866$ ai-semirings of order four, which are denoted by $S_{(4, i)}, 1\leq i\leq 866$.
One can find information about these algebras in \cite{rjzl, rlyc, yrzs}.
In the present work, we only need $S_{(4, 475)}$ and $S_{(4, 477)}$, whose Cayley tables and equational bases
are listed in Table \ref{4-element ai-semirings}.
We assume that the carrier set of each of these semirings is $\{1, 2, 3, 4\}$.
It is straightforward to check that $S_{(4, 475)}$ is in the variety ${\bf R}$.

\begin{table}[htbp]
\caption{Some 4-element ai-semirings}\label{4-element ai-semirings}
\begin{tabular}{cccc}
\hline
Semiring&   $+$ &  $\cdot$ & Equational basis\\
\hline
$S_{(4,475)}$&   \begin{tabular}{cccc}
                   1 & 1 & 1 & 1 \\
                   1 & 2 & 3 & 4\\
                   1 & 3 & 3 & 1\\
                   1 & 4 & 1 & 4
              \end{tabular}&
\begin{tabular}{cccc}
                    3 & 3 & 3 & 3 \\
                    2 & 2 & 2 & 2\\
                    3 & 3 & 3 & 3\\
                    3 & 3 & 3 & 3
              \end{tabular}
&
\begin{tabular}{cccc}
                    & $xy\approx xz$
\end{tabular}
\\
\hline
$S_{(4,477)}$&   \begin{tabular}{cccc}
                   1 & 1 & 1 & 1 \\
                   1 & 2 & 3 & 4\\
                   1 & 3 & 3 & 1\\
                   1 & 4 & 1 & 4
              \end{tabular}&
\begin{tabular}{cccc}
                    3 & 2 & 3 & 3 \\
                    3 & 2 & 3 & 3\\
                    3 & 2 & 3 & 3\\
                    3 & 2 & 3 & 3
              \end{tabular}
&
\begin{tabular}{cccc}
                    & $xy\approx zy$
\end{tabular}
\\
\hline
\end{tabular}

\end{table}

Let $K$ be a class of ai-semirings. Then $\mathsf{V}(K)$ denotes the variety generated by $K$.
In what follows we present concise and explicit equational bases for $\mathsf{V}(L_2, T_2)$, $\mathsf{V}(L_2, N_2)$,
$\mathsf{V}(N_2, T_2)$, and $\mathsf{V}(L_2, N_2, T_2)$, even though all of them have already been shown to be finitely based in \cite{sr}.

\begin{pro}\label{prolt}
$\mathsf{V}(L_2, T_2)$ is the subvariety of ${\bf R}$ defined by the identities
\begin{align}
x^2 & \approx x^2+x; \label{lt02}\\
x+y^2 &\approx x^2+y^2. \label{lt03}
\end{align}
\end{pro}
\begin{proof}
It is easy to check that both $L_2$ and $T_2$ satisfy the identities \eqref{lt02} and \eqref{lt03}.
In the remainder it is enough to prove that every ai-semiring identity that holds in both $L_2$ and $T_2$
can be derived by \eqref{lt02} and \eqref{lt03} and the identities defining ${\bf R}$.
Let $\bu \approx \bu+\bq$ be such a nontrivial identity,
where $\bu=\bu_1+\bu_2+\cdots+\bu_n$ with $\bu_i, \bq\in X^+$ for $1 \leq i \leq n$.
Then there exist $\bu_i, \bu_j\in \bu$ such that $h(\bu_i)=h(\bq)$ and $\ell(\bu_j)\geq2$.

If $\ell(\mathbf{q})=1$, then $\mathbf{u}_i=\mathbf{q}s(\mathbf{u}_i)$, and so
\[
\bu \approx \bu+\bu_i \approx \bu+\bq s(\mathbf{u}_i) \stackrel{\eqref{id0703}}\approx \bu+\bq^2\stackrel{\eqref{lt02}}\approx \bu+\bq^2+\bq.
\]
This implies the identity $\bu \approx \bu+\bq$.
Now suppose that $\ell(\bq)\geq2$. Then $\bq=h(\bq)s(\bq)$.
Consider the following two cases.

{\bf Case 1}. $\ell(\bu_i)\geq2$.  Then $\bu_i=h(\bq)s(\bu_i)$, where $s(\bu_i)$ is nonempty. Now we have
\[
\bu \approx \bu+\bu_i \approx \bu+h(\bq)s(\bu_i) \stackrel{\eqref{id0703}}\approx \bu+h(\bq)s(\bq)\approx  \bu+\bq.
\]

{\bf Case 2}. $\ell(\bu_i)=1$. Then $\bu_i=h(\bq)$. So we derive
\begin{align*}
\bu
&\approx \bu+\bu_i+\bu_j\\
&\approx \bu+\bu_i+h(\bu_j)s(\bu_j)\\
&\approx \bu+h(\bq)+h(\bu_j)h(\bu_j)&&(\text{by}~\eqref{id0703})\\
&\approx \bu+h(\bq)^2+h(\bu_j)h(\bu_j)&&(\text{by}~\eqref{lt03})\\
&\approx \bu+h(\bq)s(\bq)+h(\bu_j)h(\bu_j)&&(\text{by}~\eqref{id0703})\\
&\approx \bu+\bq+h(\bu_j)h(\bu_j).
\end{align*}
This implies the identity $\bu\approx \bu+\bq$.
\end{proof}

\begin{cor}\label{meet}
$\mathsf{V}(L_2, T_2)$ is the subvariety of $\mathsf{V}(S_{58})$ defined by the identity \eqref{lt03}.
\end{cor}

\begin{pro}\label{proln}
$\mathsf{V}(L_2, N_2)$ is the subvariety of ${\bf R}$ defined by the identity
\begin{align}
x & \approx x+xy. \label{ln02}
\end{align}
\end{pro}
\begin{proof}
It is easy to verify that both $L_2$ and $N_2$ satisfy the identity \eqref{ln02}.
In what follows, it suffices to show that every ai-semiring identity that holds in both $L_2$ and $N_2$
can be derived by \eqref{ln02} and the identities defining $\bf{R}$.
Let $\bu \approx \bu+\bq$ be such a nontrivial identity,
where $\bu=\bu_1+\bu_2+\cdots+\bu_n$ with $\bu_i, \bq\in X^+$ for $1 \leq i \leq n$.
Then $\ell(\bq)\geq2$ and there exists $\bu_i\in \bu$ such that $h(\bu_i)=h(\bq)$.

If $\ell(\bu_i)\geq2$, then $\bu_i=h(\bq)s(\bu_i)$, where $s(\bu_i)$ is nonempty. So we have
\[
\bu \approx \bu+\bu_i \approx \bu+h(\bq)s(\bu_i) \stackrel{\eqref{id0703}}\approx \bu+h(\bq)s(\bq)\approx  \bu+\bq.
\]
If $\ell(\bu_i)=1$, then $\bu_i=h(\bq)$, and so
\[
\bu \approx \bu+\bu_i \approx \bu+h(\bq) \stackrel{\eqref{ln02}}\approx \bu+h(\bq)+h(\bq)s(\bq) \approx \bu+h(\bq)+\bq \approx  \bu+\bq.
\]
This completes the proof.
\end{proof}

\begin{pro}\label{pront}
$\mathsf{V}(N_2, T_2)$ is the ai-semiring variety defined by the identity
\begin{align}
x_1x_2 & \approx y_1y_2. \label{nt01}
\end{align}
\end{pro}
\begin{proof}
It is straightforward to verify that both $N_2$ and $T_2$ satisfy the identity \eqref{nt01}.
To complete the proof, we need to show that every ai-semiring identity satisfied by both $N_2$ and $T_2$
can be derived by \eqref{nt01} and the identities defining the variety of all ai-semirings.
Let $\bu \approx \bu+\bq$ be such a nontrivial identity,
where $\bu=\bu_1+\bu_2+\cdots+\bu_n$ with $\bu_i, \bq\in X^+$ for $1 \leq i \leq n$.
Then $\ell(\bq)\geq2$ and $\ell(\bu_i)\geq2$ for some $\bu_i\in \bu$.
We then derive
\[
\bu \approx \bu+\bu_i \stackrel{\eqref{nt01}}\approx \bu+\bq.
\]
This completes the proof.
\end{proof}

\begin{pro}\label{prolnt}
$\mathsf{V}(L_2, N_2, T_2)$ is  the subvariety of ${\bf R}$ defined by the identity
\begin{align}
x+y^2\approx x+y^2+x^2. \label{lnt02}
\end{align}
\end{pro}
\begin{proof}
It is easy to verify that $L_2$, $N_2$, and $T_2$ satisfy the identity \eqref{lnt02}.
In the remainder it is enough to prove that every ai-semiring identity holding $L_2$, $N_2$ and $T_2$
can be derived by \eqref{lnt02} and the identities defining $\mathbf{R}$.
Let $\bu \approx \bu+\bq$ be such a nontrivial identity,
where $\bu=\bu_1+\bu_2+\cdots+\bu_n$ with $\bu_i, \bq\in X^+$ for $1 \leq i \leq n$.
Then $\ell(\bq)\geq2$, $h(\bu_i)=h(\bq)$ and $\ell(\bu_j)\geq2$ for some $\bu_i, \bu_j\in \bu$.

If $\ell(\bu_i)\geq2$, then $\bu_i=h(\bq)s(\bu_i)$, where $s(\bu_i)$ is nonempty.
So we have
\[
\bu \approx \bu+\bu_i \approx \bu+h(\bq)s(\bu_i) \stackrel{\eqref{id0703}}\approx \bu+h(\bq)s(\bq)\approx  \bu+\bq.
\]
If $\ell(\mathbf{u}_i)=1$, then $\mathbf{u}_i=h(\mathbf{q})$. So we obtain
\begin{align*}
\bu
&\approx \bu+\bu_i+\bu_j\\
&\approx \bu+\bu_i+h(\bu_j)s(\bu_j)\\
&\approx \bu+h(\bq)+h(\bu_j)h(\bu_j)&&(\text{by}~\eqref{id0703})\\
&\approx \bu+h(\bq)+h(\bu_j)h(\bu_j)+h(\bq)^2&&(\text{by}~\eqref{lnt02})\\
&\approx \bu+h(\bq)+h(\bu_j)h(\bu_j)+h(\bq)s(\bq)&&(\text{by}~\eqref{id0703})\\
&\approx \bu+h(\bq)+h(\bu_j)h(\bu_j)+\bq.
\end{align*}
This derives the identity $\bu \approx \bu+\bq$.
\end{proof}

\section{The subvariety lattice of the variety $\mathbf{R}$}
In this section we characterize the subvariety lattice of the variety $\mathbf{R}$.
The following result, which is due to Yue et al. \cite[Proposition 4.5, Remark 4.6]{yrzs},
shows that the variety $\mathbf{R}$ is finitely generated.
\begin{lem}\label{lem070201}
Both $S_{(4, 475)}$ and $\{S_{58}, N_2\}$ are generating sets for the variety $\mathbf{R}$.
\end{lem}

Next, we identify the minimal nontrivial subvarieties of $\mathbf{R}$.

\begin{pro}\label{LTN}
$\mathsf{V}(L_2)$, $\mathsf{V}(N_2)$ and $\mathsf{V}(T_2)$ are the only minimal nontrivial subvarieties of $\mathbf{R}$.
\end{pro}
\begin{proof}
From \cite[Theorem 1.1]{p} and \cite{sr}, we know that $\mathsf{V}(L_2)$, $\mathsf{V}(R_2)$, $\mathsf{V}(M_2)$, $\mathsf{V}(D_2)$,
$\mathsf{V}(N_2)$, and $\mathsf{V}(T_2)$ form a complete list of the minimal nontrivial subvarieties of the
variety of all ai-semirings. It is easy to see that $L_2$, $N_2$, and $T_2$ satisfy the identity \eqref{id0703},
but none of $R_2$, $M_2$, and $D_2$ satisfy \eqref{id0703}.
Therefore, $\mathsf{V}(L_2)$, $\mathsf{V}(N_2)$, and $\mathsf{V}(T_2)$ are the only minimal nontrivial subvarieties of $\mathbf{R}$.
\end{proof}

\begin{pro}\label{L2}
Let $\mathcal{V}$ be a subvariety of $\mathbf{R}$.
Then $\mathcal{V}$ does not contain $L_2$ if and only if $\mathcal{V}$ satisfies the identity
\begin{align}
x^2\approx x^2+y^2.\label{L}
\end{align}
\end{pro}
\begin{proof}
Suppose that $\mathcal{V}$ satisfies \eqref{L}.
Since the identity \eqref{L} does not hold in $L_2$,
it follows immediately that $\mathcal{V}$ does not contain $L_2$.
Conversely, assume that $\mathcal{V}$ does not satisfy the identity \eqref{L}.
Then there exists a semiring $S$ in $\mathcal{V}$ such that $a^2\neq a^2+b^2$ for some $a, b\in S$.
One can easily verify that $\{a^2, a^2+b^2\}$ is isomorphic to $L_2$.
Thus $\mathcal{V}$ contains $L_2$ as required.
\end{proof}

\begin{pro}\label{N2}
Let $\mathcal{V}$ be a subvariety of $\mathbf{R}$.
Then $\mathcal{V}$ does not contain $N_2$ if and only if $\mathcal{V}$ satisfies the identity
\begin{align}
x^2\approx x^2+x.\label{N}
\end{align}
\end{pro}
\begin{proof}
Suppose that $\mathcal{V}$ satisfies \eqref{N}.
Since the identity \eqref{N} is not satisfied by $N_2$, we have that $\mathcal{V}$ does not contain $N_2$.
Now assume that $\mathcal{V}$ does not satisfy the identity \eqref{N}.
Then there exists a member $S$ of $\mathcal{V}$ such that $a^2\neq a^2+a$ for some $a\in S$.
This implies that $\{a^2, a^2+a\}$ is isomorphic to $N_2$.
Hence $\mathcal{V}$ contains $N_2$.
\end{proof}

\begin{pro}\label{T2}
Let $\mathcal{V}$ be a subvariety of $\mathbf{R}$.
Then $\mathcal{V}$ does not contain $T_2$ if and only if $\mathcal{V}$ satisfies the identity
\begin{align}
x\approx x+x^2.\label{T}
\end{align}
\end{pro}
\begin{proof}
Suppose that $\mathcal{V}$ satisfies \eqref{T}.
Since this identity is not true in $T_2$, it follows that $\mathcal{V}$ does not contain $T_2$.
Conversely, assume that $\mathcal{V}$ does not satisfy the identity \eqref{T}.
Then there exists a semiring $S$ in $\mathcal{V}$ such that $a\neq a+a^2$ for some $a\in S$, and so $a^2\nleq a$.
If $a \leq a^2$, then $\{a^2, a\}$ is isomorphic to $T_2$.
If $a\nleq a^2$, then $\{a^2, a, a^2+a\}$ forms a $3$-element subalgebra of $S$.
Let $I$ denote the set $\{a^2, a^2+a\}$.
It is easy to see that $I$ is
both a multiplicative ideal and an order-theoretic filter of $\{a^2, a, a^2+a\}$.
Furthermore, the quotient algebra
$\{a^2, a, a^2+a\}/I$ is isomorphic to $T_2$. Thus $\mathcal{V}$ contains $T_2$.
\end{proof}


Let $\mathsf{V}(L_2, N_2, T_2)$ denote the variety generated by $L_2$, $N_2$ and $T_2$.
From \cite{sr} we know that the lattice $\mathcal{L}(\mathsf{V}(L_2, N_2, T_2))$ of subvarieties of $\mathsf{V}(L_2, N_2, T_2)$
contains exactly $8$ varieties: $\mathsf{V}(L_2, N_2, T_2)$, $\mathsf{V}(L_2, N_2)$, $\mathsf{V}(L_2, T_2)$,
$\mathsf{V}(N_2, T_2)$, $\mathsf{V}(L_2)$, $\mathsf{V}(T_2)$, $\mathsf{V}(N_2)$,
and the trivial variety $\mathbf{T}$, which are all finitely based.
To characterize the lattice $\mathcal{L}(\mathbf{R})$ of subvarieties of $\mathbf{R}$,
it is natural to consider the mapping
\[
\varphi: \mathcal{L}(\mathbf{R})\rightarrow \mathcal{L}(\mathsf{V}(L_2, N_2, T_2)),
~\mathcal{V}\mapsto \mathcal{V}\cap \mathsf{V}(L_2, N_2, T_2).
\]
Then $\varphi$ is surjective, and so $\mathcal{L}(\mathbf{R})$ is the disjoint union of
$\varphi^{-1}(\mathbf{T})$, $\varphi^{-1}(\mathsf{V}(L_2))$, $\varphi^{-1}(\mathsf{V}(N_2))$,
$\varphi^{-1}(\mathsf{V}(T_2))$, $\varphi^{-1}(\mathsf{V}(L_2, N_2))$, $\varphi^{-1}(\mathsf{V}(T_2, N_2))$,
$\varphi^{-1}(\mathsf{V}(L_2, T_2))$, and $\varphi^{-1}(\mathsf{V}(L_2, N_2, T_2))$.

\begin{pro}\label{8}
\hspace*{\fill}
\begin{itemize}
\item[$(1)$] $\varphi^{-1}(\mathbf{T})=\{\mathbf{T}\}$.

\item[$(2)$] $\varphi^{-1}(\mathsf{V}(L_2))=\{\mathsf{V}(L_2)\}$.

\item[$(3)$] $\varphi^{-1}(\mathsf{V}(N_2))=\{\mathsf{V}(N_2)\}$.

\item[$(4)$] $\varphi^{-1}(\mathsf{V}(T_2))=\{\mathsf{V}(T_2)\}$.

\item[$(5)$] $\varphi^{-1}(\mathsf{V}(L_2, N_2))=\{\mathsf{V}(L_2, N_2)\}$.

\item[$(6)$] $\varphi^{-1}(\mathsf{V}(N_2, T_2))=\{\mathsf{V}(N_2, T_2)\}$.

\item[$(7)$] $\varphi^{-1}(\mathsf{V}(L_2, T_2))=\{\mathsf{V}(L_2, T_2), \mathsf{V}(S_{58})\}$.

\item[$(8)$] $\varphi^{-1}(\mathsf{V}(L_2, N_2, T_2))=\{\mathsf{V}(L_2, N_2, T_2), \mathbf{R}\}$.
\end{itemize}
\end{pro}

\begin{proof}
$(1)$ This follows from proposition \ref{LTN} immediately.

$(2)$ It is easy to see that $\mathsf{V}(L_2)$ is a member of $\varphi^{-1}(\mathsf{V}(L_2))$.
Now let $\mathcal{V}$ be an arbitrary variety in $\varphi^{-1}(\mathsf{V}(L_2))$.
Then $\mathcal{V}$ contains $L_2$, but does not contain $N_2$ or $T_2$.
By Propositions \ref{N2} and \ref{T2}, we obtain that $\mathcal{V}$ satisfies the identities \eqref{N} and \eqref{T},
and so $\mathcal{V}$ also satisfies the identity $x^2 \approx x$.
By using the identity \eqref{id0703},
it follows that $\mathcal{V}$ satisfies $xy\approx x$,
which is an equational basis of $\mathsf{V}(L_2)$.
Thus $\mathcal{V}=\mathsf{V}(L_2)$ and so $\varphi^{-1}(\mathsf{V}(L_2))=\{\mathsf{V}(L_2)\}$.

$(3)$  It is immediate that that $\mathsf{V}(N_2)$ is a member of $\varphi^{-1}(\mathsf{V}(N_2))$.
Now suppose that $\mathcal{V}$ is an arbitrary variety in $\varphi^{-1}(\mathsf{V}(N_2))$.
Then $\mathcal{V}$ contains $N_2$, but does not contain $L_2$ or $T_2$.
By Propositions~\ref{L2} and \ref{T2}, $\mathcal{V}$ satisfies the identities \eqref{L} and \eqref{T}.
Furthermore, we establish the following derivation:
\[
x_1x_2\stackrel{\eqref{id0703}}\approx x_1x_1 \stackrel{\eqref{L}}\approx x_1x_1+y_1y_1
\approx y_1y_1+x_1x_1 \stackrel{\eqref{L}}\approx y_1y_1\stackrel{\eqref{id0703}}\approx y_1y_2.
\]
Consequently,  $\mathcal{V}$ satisfies the identities $x_1x_2\approx y_1y_2$ and \eqref{T},
which form an equational basis of $\mathsf{V}(N_2)$.
Hence $\mathcal{V}=\mathsf{V}(N_2)$ and so $\varphi^{-1}(\mathsf{V}(N_2))=\{\mathsf{V}(N_2)\}$.

$(4)$ We first observe that $\mathsf{V}(T_2)$ lies in $\varphi^{-1}(\mathsf{V}(T_2))$.
Consider an arbitrary variety $\mathcal{V}$ in $\varphi^{-1}(\mathsf{V}(T_2))$.
Then $\mathcal{V}$ contains $T_2$, but does not contain $L_2$ or $N_2$.
It follows from Propositions \ref{L2} and \ref{N2} that
$\mathcal{V}$ satisfies the identities \eqref{L} and \eqref{N}.
Through the identity \eqref{id0703}, one can deduce that
$\mathcal{V}$ satisfies the identities $x_1x_2\approx y_1y_2$ and \eqref{N},
which constitute an equational basis of $\mathsf{V}(T_2)$.
We conclude that $\mathcal{V}=\mathsf{V}(T_2)$.
Consequently, $\varphi^{-1}(\mathsf{V}(T_2))=\{\mathsf{V}(T_2)\}$.

$(5)$
It is evident that $\mathsf{V}(L_2, N_2)$ is in $\varphi^{-1}(\mathsf{V}(L_2, N_2))$.
Now let $\mathcal{V}$ be an arbitrary variety in $\varphi^{-1}(\mathsf{V}(L_2, N_2))$.
Then $\mathcal{V}$ contains $L_2$ and $N_2$, but does not contain $T_2$.
By Proposition~\ref{T2} we have that $\mathcal{V}$ satisfies \eqref{T}.
Furthermore, using the identity \eqref{id0703}, we conclude that $\mathcal{V}$ satisfies the identity \eqref{ln02}.
It then  follows from Proposition \ref{proln} that
$\mathcal{V}=\mathsf{V}(L_2, N_2)$. Therefore, $\varphi^{-1}(\mathsf{V}(L_2, N_2))=\{\mathsf{V}(L_2, N_2)\}$.

$(6)$
It is easy to see that $\mathsf{V}(N_2, T_2)$ is in $\varphi^{-1}(\mathsf{V}(N_2, T_2))$.
Now let $\mathcal{V}$ be an arbitrary variety in $\varphi^{-1}(\mathsf{V}(N_2, T_2))$.
Then $\mathcal{V}$ contains $N_2$ and $T_2$, but does not contain $L_2$.
By Proposition~\ref{L2}, the variety $\mathcal{V}$ satisfies the identity \eqref{L}.
Furthermore, through the identity \eqref{id0703}, one can establish that $\mathcal{V}$ satisfies the identity \eqref{nt01}.
Applying Proposition \ref{pront}, we conclude that
$\mathcal{V}=\mathsf{V}(N_2, T_2)$.
Thus $\varphi^{-1}(\mathsf{V}(N_2, T_2))=\{\mathsf{V}(N_2, T_2)\}$.

$(7)$
First, we can directly verify that both $\mathsf{V}(L_2, T_2)$ and $\mathsf{V}(S_{58})$ belong to
the set $\varphi^{-1}(\mathsf{V}(L_2, T_2))$.
Now consider an arbitrary variety $\mathcal{V}$ in $\varphi^{-1}(\mathsf{V}(L_2, T_2))$.
Then $\mathcal{V}$ contains $L_2$ and $T_2$, but does not contain $N_2$.
Proposition \ref{N2} immediately yields that $\mathcal{V}$ satisfies the identity \eqref{N}.
We conclude that $\mathcal{V}$ lies in the interval $[\mathsf{V}(L_2, T_2), \mathsf{V}(S_{58})]$.
If $\mathcal{V}\neq\mathsf{V}(L_2, T_2)$, then by Corollary \ref{meet} that
$\mathcal{V}$ does not satisfy the identity \eqref{lt03}.
So there exists a semiring $S$ in $\mathcal{V}$ such that $a+b^2\neq a^2+b^2$ for some $a, b\in S$.
Through careful application of identities \eqref{id0703} and \eqref{N},
one can establish that $\{a^2+b^2, a+b^2, b^2\}$ is isomorphic to $S_{58}$.
Thus $\mathcal{V}=\mathsf{V}(S_{58})$ and so
$\varphi^{-1}(\mathsf{V}(L_2, T_2))=\{\mathsf{V}(L_2, T_2), \mathsf{V}(S_{58})\}$.

$(8)$
It is obvious that both $\mathsf{V}(L_2, N_2, T_2)$ and $\mathbf{R}$ are in $\varphi^{-1}(\mathsf{V}(L_2, T_2, N_2))$.
Now let $\mathcal{V}$ be an arbitrary variety in $\varphi^{-1}(\mathsf{V}(L_2, T_2, N_2))$.
Then $\mathcal{V}$ contains $L_2$, $T_2$ and $N_2$.
If $\mathcal{V}$ properly contains $\mathsf{V}(L_2, T_2, N_2)$, then by
Proposition \ref{prolnt}, $\mathcal{V}$ does not satisfy the identity \eqref{lnt02}.
So there exists a semiring $S$ in $\mathcal{V}$ such that $a+b^2\neq a+a^2+b^2$ for some $a, b\in S$.
If $a^2+b^2=a+a^2+b^2$, then it is a routine matter to verify that $\{a^2+b^2, a+b^2, b^2\}$ is isomorphic to $S_{58}$.
This implies that $\mathcal{V}$ contains the join of $\mathsf{V}(L_2, N_2, T_2)$ and $\mathsf{V}(S_{58})$.
By Lemma \ref{lem070201}, we conclude that $\mathcal{V}=\mathbf{R}$.
If $a^2+b^2\neq a+a^2+b^2$, then it is easy to check that
$\{a+a^2+b^2, a^2+b^2, a+b^2, b^2\}$ is isomorphic to $S_{(4, 475)}$.
Lemma \ref{lem070201} implies that $\mathcal{V}=\mathbf{R}$.
Therefore, $\varphi^{-1}(\mathsf{V}(L_2, N_2, T_2))=\{\mathsf{V}(L_2, N_2, T_2), \mathbf{R}\}$.
\end{proof}

\begin{thm}\label{main}
The subvariety lattice of the variety $\mathbf{R}$ is a distributive lattice of order $10$,
whose Hasse diagram is shown in Figure~$\ref{figure1}$.

\setlength{\unitlength}{1.5cm}
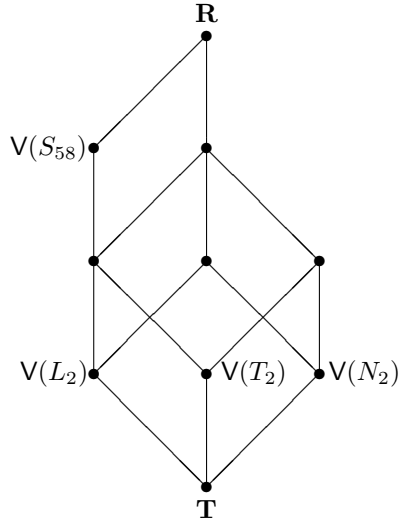
\begin{figure}[ht]
\begin{picture}(6,4.3)
\put(3.2,1.2){\line(1,1){1}}
\put(3.2,1.2){\line(-1,1){1}}
\put(3.2,2.2){\line(0,1){1.1}}
\put(3.2,3.3){\line(0,1){0.94}}
\put(3.2,3.2){\line(1,-1){1}}
\put(3.2,4.2){\line(-1,-1){1}}
\put(3.2,3.2){\line(-1,-1){1}}
\put(3.2,1.2){\line(0,-1){1}}
\put(2.2,2.2){\line(0,-1){0.97}}
\put(4.2,2.2){\line(0,-1){0.97}}
\put(2.2,2.2){\line(0,1){0.97}}
\put(3.2,2.2){\line(-1,-1){1}}
\put(3.2,2.2){\line(1,-1){1}}
\put(2.2,1.2){\line(1,-1){1}}
\put(4.2,1.2){\line(-1,-1){1}}
\multiput(3.2,1.2)(-1.5,0){1}{\circle*{0.1}}
\multiput(4.2,2.2)(1.5,0){1}{\circle*{0.1}}
\multiput(3.2,2.2)(1.5,0){1}{\circle*{0.1}}
\multiput(2.2,2.2)(1.5,0){1}{\circle*{0.1}}
\multiput(4.2,1.2)(-1,1){1}{\circle*{0.1}}
\multiput(2.2,1.2)(1,1){1}{\circle*{0.1}}
\multiput(3.2,0.2)(5,2){1}{\circle*{0.1}}
\multiput(3.2,3.2)(5,2){1}{\circle*{0.1}}
\multiput(2.2,3.2)(1.5,0){1}{\circle*{0.1}}
\multiput(3.2,4.2)(5,2){1}{\circle*{0.1}}
\put(3.2,0){\makebox(0,0){$\mathbf{T}$}}
\put(3.62,1.2){\makebox(0,0){$\mathsf{V}(T_2)$}}
\put(4.6,1.2){\makebox(0,0){$\mathsf{V}(N_2)$}}
\put(1.85,1.2){\makebox(0,0){$\mathsf{V}(L_2)$}}
\put(3.2,4.4){\makebox(0,0){$\mathbf{R}$}}
\put(1.8,3.2){\makebox(0,0){$\mathsf{V}(S_{58})$}}
\end{picture}
\caption{The subvariety lattice of $\mathbf{R}$.}\label{figure1}
\end{figure}
\end{thm}

\begin{proof}
By Proposition~\ref{8}, the subvariety lattice of the variety $\mathbf{R}$ consists of $10$ varieties,
whose Hasse diagram is given in Figure~$\ref{figure1}$.
Moreover, one can readily verify that this lattice is distributive.
\end{proof}

\begin{cor}\label{coro}
The variety $\mathbf{R}$ is a Cross variety.
Therefore, every subvariety of $\mathbf{R}$ is both finitely based and finitely generated.
\end{cor}
\begin{proof}
This is a consequence of Lemma \ref{lem070201} and Theorem \ref{main}.
\end{proof}

Let $\mathbf{C}$ denote the ai-semiring variety defined by the identity $yx\approx zx$.
Then every member of $\mathbf{C}$ has dual multiplication structure corresponding to some algebra in $\mathbf{R}$.
By Theorem~\ref{main} and Corollary \ref{coro} we immediately deduce
\begin{cor}\label{coro1}
The variety $\mathbf{C}$ is a Cross variety whose subvariety lattice is a distributive lattice
with exactly $10$ elements\up: $\mathbf{T}$, $\mathsf{V}(R_2)$, $\mathsf{V}(T_2)$,
$\mathsf{V}(N_2)$, $\mathsf{V}(R_2, N_2)$, $\mathsf{V}(R_2, T_2)$, $\mathsf{V}(N_2, T_2)$,
$\mathsf{V}(R_2, N_2, T_2)$, $\mathsf{V}(S_{56})$, and $\mathbf{C}$ itself.
Consequently, every subvariety of the variety $\mathbf{C}$ is both finitely based and finitely generated.
\end{cor}

\section{Conclusion}
We have demonstrated that every ai-semiring satisfying the identity $xy \approx xz$ or $yx \approx zx$
has a finite basis for its equational theory. A systematic verification reveals that there are precisely 789 such algebras among all ai-semirings of order less than six. The main result of this paper thus provides fundamental insights into the finite basis problem for small-order ai-semirings, establishing a rigorous foundation for future research in this area. In particular, our findings suggest promising directions for investigating the finite basis property in broader classes of ai-semirings.

\qquad

\noindent
\textbf{Acknowledgements}
The authors thank their team members Zidong Gao, Simin Lyu, Chenyu Yang and Ting Yu for discussions contributed to this paper.

\bibliographystyle{amsplain}


\end{document}